\documentclass[a4paper,reqno]{amsart}
\usepackage{geometry}
\usepackage{mathrsfs}
\usepackage{amsmath,amsthm,amsfonts,amssymb}
\usepackage{enumerate}
\usepackage{needspace}
\usepackage{xcolor}
\usepackage[colorlinks=true,linkcolor=blue,citecolor=blue]{hyperref}
\numberwithin{equation}{section}

\theoremstyle{plain}
\newtheorem{theorem}{Theorem}[section]
\newtheorem{proposition}[theorem]{Proposition}
\newtheorem{lemma}[theorem]{Lemma}
\newtheorem{corollary}[theorem]{Corollary}
\theoremstyle{definition}
\newtheorem{definition}[theorem]{Definition}
\newtheorem{example}[theorem]{Example}
\newtheorem{remark}[theorem]{Remark}

\newcommand{\C}{\mathcal C}
\newcommand{\T}{\mathcal T}
\newcommand{\X}{\mathcal X}
\newcommand{\A}{\mathcal A}
\newcommand{\E}{\mathbb E}
\newcommand{\Hom}{\operatorname{Hom}}
\newcommand{\add}{\operatorname{add}}
\newcommand{\id}{\operatorname{id}}
\newcommand{\op}{\mathrm{op}}
\newcommand{\zero}{\mathbf 0}
\newcommand{\SigX}{\Sigma_{\X}}
\newcommand{\OmX}{\Omega_{\X}}
\newcommand{\QX}{\C/[\X]}
\newcommand{\piX}{\pi_{\X}}

\title[Higher orthogonality and truncation in canonical pretriangulated quotients]{Higher orthogonality and truncation in canonical pretriangulated quotients}

\hypersetup{
  pdftitle={Higher orthogonality and truncation in canonical pretriangulated quotients},
  pdfauthor={Yixia Zhang and Panyue Zhou},
  pdfkeywords={extriangulated category, ideal quotient, pretriangulated category, higher maximal orthogonality, cluster tilting subcategory}
}

\author{Yixia Zhang and Panyue Zhou}

\address{School of Mathematics and Statistics, Changsha University of Science and Technology, 410114 Changsha, Hunan, P. R. China}
\email{yxzhangmath@163.com}

\address{School of Mathematics and Statistics, Changsha University of Science and Technology, 410114 Changsha, Hunan, P. R. China}
\email{panyuezhou@163.com}
\date{}

\makeatletter
\@namedef{subjclassname@2020}{\textup{2020} Mathematics Subject Classification}
\makeatother

\begin{document}

\begin{abstract}
Canonical ideal quotients of extriangulated categories admit natural
one-sided triangulated or pretriangulated structures. A fundamental question
is whether these quotient structures still retain enough information to recover
higher orthogonality properties of the subcategory being factored out. We show
that, for a strongly functorially finite $n$-rigid subcategory, such information
is encoded by the nilpotency of the canonical suspension and, equivalently, of
the canonical loop. More precisely, this nilpotency characterizes two-sided
maximal $n$-orthogonality. We further establish an objectwise recognition
criterion that combines one-sided higher orthogonality with the vanishing of
the $n$th suspension or loop. In the triangulated setting, our results give a
converse to the known truncation construction and show that the quotient is
$n$-truncated if and only if the subcategory is $(n+1)$-cluster tilting.
Moreover, in the nonsplit case, the nilpotency index is exactly $n$. These
results turn truncation from a consequence of higher orthogonality into a sharp
criterion for recognizing it.
\end{abstract}
\maketitle

\section{Introduction}

Ideal quotients are a basic bridge between approximation theory and
homological structures. Let $(\C,\E,\mathfrak s)$ be an extriangulated
category in the sense of Nakaoka--Palu \cite{NP}, and let $\X$ be an additive
subcategory of $\C$. When $\X$ is strongly functorially finite, Zhou--Zhu
\cite[Lemma 3.21]{ZZtri} equip the ideal quotient $\C/[\X]$ with a canonical
pretriangulated structure. Here pretriangulated is used in the sense of
Beligiannis--Reiten \cite[Chapter II]{BR}. We denote its suspension and loop functors by $\SigX$ and $\OmX$, respectively.
Two established results describe the two extreme behaviours of this quotient.
At one endpoint, cluster tilting subcategories yield abelian quotients under
the standard additional hypotheses. This follows from results of Zhou--Zhu
\cite[Corollary 3.5]{ZZct} and Liu--Zhou \cite{LZabel}, extending the
triangulated theorem of Koenig--Zhu \cite[Theorem 3.3]{KZ}. At the stable
endpoint, Zhou--Zhu \cite[Theorem 4.3]{ZZtri} show, under
Auslander--Reiten duality hypotheses, that the quotient is triangulated
precisely when the ambient pair is an $\X$-mutation pair. In a triangulated
category, this specializes to the criterion $\tau\X=\X$ of J\o rgensen
\cite[Theorem 3.3]{Jor}. Thus the canonical suspension is zero at the cluster
tilting endpoint and invertible at the stable endpoint.
\vspace{1mm}

Finite nilpotency lies naturally between these two extremes. A
pretriangulated category is $n$-truncated when its $n$th suspension, or
equivalently its $n$th loop, is zero, following
Mochizuki--Nakaoka--Ogawa \cite[Definition 1.2.3]{MNO}. These authors
\cite[Theorem 2.4.15]{MNO} prove that hearts of $n$-cotorsion pairs are
abelian and $n$-truncated. Their Example 2.1.5(2) identifies the higher cluster
tilting specialization with the ideal quotient by the cluster tilting
subcategory. This forward construction shows that higher orthogonality
produces a truncated quotient. It leaves open three converse questions: whether truncation of the canonical
suspension can recognize the higher orthogonality of $\X$, whether the
recognition can be made objectwise, and whether the truncation exponent is
optimal.
\vspace{1mm}

We answer all three questions. We use the positive higher extensions
$\E^i$ constructed by Gorsky--Nakaoka--Palu
\cite[Theorem 3.6]{GNP}. A subcategory is $n$-rigid when
$\E^i(\X,\X)=0$ for $1\leq i\leq n$. We use the standard orthogonality
classes $\X^{\perp_i}$ and ${}^{\perp_i}\X$. Let $\piX\colon\C\to\C/[\X]$ denote the quotient functor. If the corresponding canonical suspension or loop exists, set
\[
 \operatorname{Tr}^{+}_{n}(\X)
 =\{A\in\C\mid \SigX^n(\piX A)\cong0\},
\hspace{2.5mm}
 \operatorname{Tr}^{-}_{n}(\X)
 =\{A\in\C\mid \OmX^n(\piX A)\cong0\}.
\]
These are local truncation loci. They record vanishing at one object rather
than on the whole quotient.

Throughout this article, we fix a
commutative ring $R$ with $1$ and a positive integer $n\geq 1$.

\begin{theorem}\label{thm:intro-local}
Let $(\C,\E,\mathfrak s)$ be a small $R$-linear extriangulated category and $\X=\add\X$ be an $n$-rigid subcategory of $\C$.
\begin{enumerate}[{\rm(1)}]
\item If $\X$ is strongly covariantly finite, then
\[
 \X
 =
 \operatorname{Tr}^{+}_{n}(\X)
 \cap
 \bigcap_{i=1}^{n}\X^{\perp_i}.
\]
\item If $\X$ is strongly contravariantly finite, then
\[
 \X
 =
 \operatorname{Tr}^{-}_{n}(\X)
 \cap
 \bigcap_{i=1}^{n}{}^{\perp_i}\X.
\]
\end{enumerate}
\end{theorem}
In Theorems~\ref{thm:local-reconstruction} and~\ref{thm:local-reconstruction-dual}, we
give the precise statements and proofs. The key observation is that the
$n$th suspension or loop needs to vanish only at the object under
consideration, rather than on the entire quotient. In Example~\ref{ex:local-not-global},
we show that this distinction is essential.

Passing from local to global vanishing gives the recognition theorem.

\begin{theorem}\label{thm:intro-main}
Let $(\C,\E,\mathfrak s)$ be a small $R$-linear extriangulated category and
 $\X=\add\X$ be a strongly functorially finite $n$-rigid
subcategory of $\C$. Then the following conditions are equivalent:
\begin{enumerate}[{\rm(1)}]
\item
\[
\X=
\bigcap_{i=1}^{n}{}^{\perp_i}\X
=
\bigcap_{i=1}^{n}\X^{\perp_i}.
\]
\item $\SigX^{n}\cong\zero$ as endofunctors of $\C/[\X]$.
\item $\OmX^{n}\cong\zero$ as endofunctors of $\C/[\X]$.
\end{enumerate}
\end{theorem}

In Theorem~\ref{thm:recognition}, we give the full statement and proof, while
Theorem~\ref{thm:one-sided} isolates the one-sided mechanism. Each
approximation direction therefore has a distinct role before the two are
combined.

When enough projective and injective objects exist,
Gorsky--Nakaoka--Palu \cite[Corollary 3.23]{GNP} identify the positive
higher extensions with the usual higher extensions. The orthogonality
condition in Theorem~\ref{thm:intro-main} is then the standard $(n+1)$
cluster tilting condition of Herschend--Liu--Nakaoka
\cite[Definition 3.21]{HLN2}. In a triangulated category the result takes a
particularly transparent form.

\begin{corollary}\label{cor:intro-triangulated}
Let $\T$ be a triangulated category and $\X=\add\X$ be a
functorially finite $n$-rigid subcategory of $\T$. Then the canonical pretriangulated
quotient $\T/[\X]$ is $n$-truncated if and only if $\X$ is $(n+1)$ cluster
tilting.
\end{corollary}

In Corollary~\ref{cor:triangulated-ntruncated}, we give the precise statement, the objectwise formulas, and the proof. The forward implication is
contained in the construction of Mochizuki--Nakaoka--Ogawa
\cite[Example 2.1.5(2), Definition 2.3.1 and Proposition 2.3.23]{MNO}.
In that specialization their heart is $\T/[\X]$, and their suspension is
constructed from the same left $\X$-approximation triangles that define the
canonical quotient suspension used here. The converse implication is new.
Under functorial finiteness and $n$-rigidity, truncation of the canonical
quotient already forces higher cluster tilting.
\vspace{1mm}

Our final main result determines the exact nilpotency exponent.

\begin{theorem}\label{thm:intro-sharp}
Let $(\C,\E,\mathfrak s)$ be a small $R$-linear extriangulated category. Assume that $\C$ is nonsplit and that $\X$ is strongly functorially finite and two-sided maximally $n$-orthogonal. Then the
nilpotency index of $\SigX$ is exactly $n$, and
$\E^{n+1}(\X,\X)\neq0$.
\end{theorem}

In Theorem~\ref{thm:sharp-truncation}, we give the detailed statement and proof.
In particular, the positive global dimension of $\C$ is at least $n+1$ in the
sense of Gorsky--Nakaoka--Palu \cite[Definition 3.30]{GNP}. The bound is
realized explicitly in every degree in Example~\ref{ex:sharp}. For a nonzero
triangulated category, Corollary~\ref{cor:strict-triangulated} says that a
higher cluster tilting quotient is exactly $n$-truncated, not merely
$n$-truncated.

We emphasize how these results differ from the existing literature. The
construction of Mochizuki--Nakaoka--Ogawa
\cite[Example 2.1.5(2) and Proposition 2.3.23]{MNO} proceeds from higher orthogonality to
truncation, whereas our recognition theorem proves the reverse implication.
The objectwise reconstruction formulas are stronger than global
truncatedness and have no counterpart in that forward construction. The
positive extension calculus of Gorsky--Nakaoka--Palu
\cite[Theorem 3.6]{GNP} also allows the recognition theorem to be proved in
small $R$-linear extriangulated categories without assuming enough
projective or injective objects. In exact categories, Kvamme \cite[Proposition 4.4]{Kvamme} gives a criterion
in terms of two families of resolutions. Proposition~\ref{prop:exact-two-resolutions}
identifies these two families separately with the canonical suspension and
canonical loop. Related but
different quotient constructions are studied by Mochizuki
\cite{MochizukiEMC}, Zhang--Zhou--Zhu \cite{ZZZ26}, Silberberg
\cite{Silberberg26}, and Fang--Gorsky--Palu--Plamondon--Pressland
\cite{FGPPP}.

The specialization to degree one recovers the characterization
of cluster tilting in terms of vanishing suspension, without any smallness or
linearity assumptions. In Section~\ref{sec:unification}, we compare this result with
the known criterion of Zhou--Zhu involving an invertible suspension
\cite[Theorem 4.3]{ZZtri}. We recall the corresponding stable
case only for comparison and do not claim it as a new result.

The paper is organized as follows. In Section~\ref{sec:prelim}, we fix the
notation, recall positive extensions and the known quotient constructions,
and prove the elementary lemmas used later. In
Section~\ref{sec:Truncation and higher orthogonality}, we establish the local
reconstruction theorem and the recognition theorem, and we present the
explicit examples, the interpretation in exact categories, and the sharpness
result. In Section~\ref{sec:unification}, we compare the vanishing and
invertible cases and record the specialization to triangulated categories.

\section{Background and known results}\label{sec:prelim}

Throughout this paper, $R$ is a commutative ring $R$ with $1$ and all categories are additive. An
$R$-linear category has $R$-module morphism groups and $R$-bilinear
composition. Subcategories are full and closed under isomorphisms. The
notation $\X=\add\X$ means that $\X$ is closed under finite direct sums and
direct summands. A category is called \emph{small} when its objects and
morphisms form sets. All results stated for small categories also apply to
essentially small categories after passage to a skeleton.

\subsection{Extriangulated categories and positive higher extensions}

We use the terminology and notation of Nakaoka--Palu \cite{NP}. Thus $(\C,\E,\mathfrak s)$ denotes an extriangulated category and
\[
        A\xrightarrow{x}B\xrightarrow{y}C\overset{\delta}{\dashrightarrow}
\]
denotes an $\E$-triangle realizing $\delta\in\E(C,A)$. We call $\C$ \emph{nonsplit} if $\E\neq0$. An $\E$-triangle is split if and only if its extension class is zero by Nakaoka--Palu \cite[Corollary 3.5]{NP}.

For the higher theory we use the positive extension bifunctors
\[
        \E^i\colon \C^{\op}\times\C\longrightarrow R\text{-}\mathrm{Mod},
      \hspace{2.5mm} i\geq1,
\]
constructed by Gorsky--Nakaoka--Palu \cite[Theorem 3.6]{GNP}. Their theorem gives the long exact sequences in both variables attached to every $\E$-triangle. If enough projectives and injectives exist, Gorsky--Nakaoka--Palu \cite[Corollary 3.23]{GNP} identify these functors with the standard higher extensions defined by resolutions.
We will invoke these results rather than reproduce their proofs.

Following Gorsky--Nakaoka--Palu \cite[Definition 3.30]{GNP}, we use the
positive global dimension
\[
 \operatorname{pgldim}\C
 =\sup\bigl(\{i\geq1\mid\E^i\neq0\}\cup\{0\}\bigr).
\]
The adjoined $0$ fixes the value in the split case.

For classes of objects $\mathcal U,\mathcal V\subseteq\C$, the notation
$\E^i(\mathcal U,\mathcal V)=0$ means $\E^i(U,V)=0$ for every
$U\in\mathcal U$ and $V\in\mathcal V$. An object $P$ is projective if
$\E(P,-)=0$, and an object $I$ is injective if $\E(-,I)=0$.
The category has enough projectives if every object $A$ occurs in an
$\E$-triangle $K\to P\to A\dashrightarrow$ with $P$ projective. Enough
injectives is defined dually. These are the conventions of Nakaoka--Palu
\cite[Definitions 3.23 and 3.25]{NP}.

For $i\geq1$ set
\[
{}^{\perp_i}\X
=\{A\in\C\mid \E^i(A,X)=0\text{ for all }X\in\X\},
\]
and
\[
\X^{\perp_i}
=\{A\in\C\mid \E^i(X,A)=0\text{ for all }X\in\X\}.
\]

\begin{definition}\label{def:nrigid}
A subcategory $\X$ is \emph{$n$-rigid} if
\[
        \E^i(\X,\X)=0,
        \hspace{2.5mm}1\leq i\leq n.
\]
\end{definition}

\begin{definition}\label{def:maxorth}
Let $n\geq1$. A subcategory $\X=\add\X$ is said to satisfy the \emph{two-sided maximal $n$-orthogonality condition} (with respect to the positive extensions $\E^i$) if
\[
\X
=
\bigcap_{i=1}^{n}{}^{\perp_i}\X
=
\bigcap_{i=1}^{n}\X^{\perp_i}.
\]
\end{definition}

\begin{definition}\label{def:cluster-tilting}
A subcategory $\X=\add\X$ is \emph{cluster tilting} if it is strongly functorially finite and
\[
 \X=\{A\in\C\mid\E(\X,A)=0\}
 =\{A\in\C\mid\E(A,\X)=0\}.
\]
This is the convention of Zhou--Zhu \cite[Definition 2.10]{ZZct}.
\end{definition}

More generally, in a setting with the usual higher extensions, an
\emph{$(n+1)$ cluster tilting subcategory} is a functorially finite
subcategory $\X=\add\X$ satisfying
\[
 \X=\bigcap_{i=1}^{n}{}^{\perp_i}\X
   =\bigcap_{i=1}^{n}\X^{\perp_i}.
\]
This indexing agrees with Herschend--Liu--Nakaoka
\cite[Definition 3.21]{HLN2}. A subcategory is \emph{generating} if every
object admits a deflation from an object of the subcategory, and
\emph{cogenerating} if every object admits an inflation into an object of the
subcategory.

For $n=1$, Definition~\ref{def:maxorth}, together with strong functorial finiteness, is exactly Definition~\ref{def:cluster-tilting}. For $n>1$ we retain the descriptive language of maximal higher orthogonality rather than introduce a new cluster tilting terminology in the absence of additional hypotheses.

When enough projectives and injectives exist, Gorsky--Nakaoka--Palu \cite[Corollary 3.23]{GNP} identify the positive higher extensions with the usual higher extensions defined by resolutions. In that setting, after adding functorial finiteness, Definition~\ref{def:maxorth} is precisely the standard $(n+1)$ cluster tilting condition of Herschend--Liu--Nakaoka \cite[Definition 3.21]{HLN2}. Moreover, the two orthogonality equalities force every projective and every injective object to lie in $\X$. Hence ordinary functorial finiteness and strong functorial finiteness agree in this situation by Zhou--Zhu \cite[Remark 2.9]{ZZct}. In an exact category this inclusion of projectives and injectives also makes $\X$ generating and cogenerating, so the formulation agrees with the usual exact category convention. He--Zhou \cite[Definition 2.3 and Theorem 3.5]{HZ} establish the corresponding formulation in terms of $n$-cotorsion pairs.

\subsection{Ideal quotients and canonical pretriangulations}

For an additive subcategory $\X\subseteq\C$, let $[\X](A,B)$ be the subgroup of morphisms $A\to B$ that factor through an object of $\X$. The ideal quotient $\C/[\X]$ has the same objects as $\C$ and morphism groups
\[
        (\C/[\X])(A,B)=\C(A,B)/[\X](A,B).
\]
We write $\piX\colon\C\to\C/[\X]$ for the quotient functor and use $\piX A$ and $\piX f$ for quotient images. A right triangulated category and a left triangulated category are understood in the sense of Beligiannis--Reiten \cite[Chapter II]{BR}. A compatible pair of these structures is called pretriangulated. Compatibility includes an adjunction $\Sigma\dashv\Omega$ between the suspension and loop. We denote by $\zero$ the zero endofunctor, which sends every object and morphism to zero. Following Mochizuki--Nakaoka--Ogawa \cite[Definition 1.2.3]{MNO}, a pretriangulated category is $n$-truncated if $\Sigma^n\cong\zero$, equivalently if $\Omega^n\cong\zero$.

A morphism $x_A\colon A\to X_A$ with $X_A\in\X$ is a \emph{left $\X$-approximation} if every morphism from $A$ to an object of $\X$ factors through $x_A$. The subcategory $\X$ is covariantly finite if every object admits such a morphism. Right $\X$-approximations and contravariant finiteness are defined dually, and $\X$ is functorially finite if it is both covariantly and contravariantly finite.

A subcategory $\X$ is strongly covariantly finite if each $A\in\C$ admits an $\E$-triangle
\[
        A\xrightarrow{x_A}X_A\longrightarrow A^+\dashrightarrow,
        \hspace{2.5mm} X_A\in\X,
\]
with $x_A$ a left $\X$-approximation. We call these left approximation $\E$-triangles. Strong contravariant finiteness is defined dually by right approximation $\E$-triangles, and strong functorial finiteness means both. These are the definitions of Zhou--Zhu \cite[Definition 3.19]{ZZtri}. In a triangulated category, Zhou--Zhu \cite[Remark 3.20]{ZZtri} show that strong functorial finiteness agrees with ordinary functorial finiteness.

We use both the one-sided and the two-sided quotient constructions below.

\begin{theorem}\label{thm:known-onesided}
Let $\X$ be an additive subcategory of an extriangulated category $\C$.
\begin{enumerate}[{\rm(1)}]
\item If $\X$ is strongly covariantly finite, then $\C/[\X]$ carries a canonical right triangulated structure. Its suspension, denoted $\SigX$, is induced by left $\X$-approximation $\E$-triangles.
\item Dually, if $\X$ is strongly contravariantly finite, then $\C/[\X]$ carries a canonical left triangulated structure. Its loop, denoted $\OmX$, is induced by right $\X$-approximation $\E$-triangles.
\end{enumerate}
\end{theorem}

\proof Part \textup{(1)} is the $n=1$ specialization of the construction of He--He--Zhou \cite[Lemma 3.1]{HHZright}. Part \textup{(2)} follows by applying the same result to the opposite extriangulated category. In particular, the canonical suspension or loop is available under the corresponding one-sided strong finiteness assumption.  \qed

\begin{theorem}\label{thm:known-pretri}
If $\X$ is strongly functorially finite, then the two one-sided structures of Theorem~\ref{thm:known-onesided} combine to a canonical pretriangulated structure on $\C/[\X]$, and $(\SigX,\OmX)$ is an adjoint pair.
\end{theorem}

\proof This is the construction of Zhou--Zhu \cite[Lemma 3.21]{ZZtri}. See also Zheng--Cao--Wei \cite{ZCW} for a subfactor version in extriangulated categories and Beligiannis--Reiten \cite{BR} for the pretriangulated formalism. We use these known quotient constructions as black boxes.  \qed

\begin{lemma}\label{lem:stable-right}
Let $(\mathcal R,\Sigma,\Delta)$ be a right triangulated category. If $\Sigma$ is an autoequivalence, then $(\mathcal R,\Sigma,\Delta)$ is a triangulated category.
\end{lemma}

\proof This is part of the standard right triangulated formalism recalled by Lin--Wang \cite[Section 2.1.1]{LW}. We use it only for the canonical right triangulated structure on an ideal quotient.  \qed

\subsection{The two known endpoints}

We record precisely the two established quotient theorems that motivate the present work.

\begin{theorem}[Abelian endpoint]\label{thm:known-abelian}
Let $\C$ be an extriangulated category with enough projectives and enough injectives, and let $\X$ be a cluster tilting subcategory. Then $\C/[\X]$ is equivalent to $\mathrm{mod}\,(\X/[\mathcal P])$ and in particular is abelian.
\end{theorem}

\proof This is due to Zhou--Zhu \cite[Corollary 3.5]{ZZct}. The triangulated specialization of Zhou--Zhu \cite[Corollary 3.6]{ZZct} recovers the theorem of Koenig--Zhu \cite[Theorem 3.3]{KZ}. We do not repeat these known abelian quotient results.  \qed

Following Zhou--Zhu \cite[Section 4]{ZZtri}, let $\mathcal P$ and $\mathcal I$ denote the projective and injective subcategories. Write $\underline{\X}$ for the full subcategory of $\C/[\mathcal P]$ consisting of objects isomorphic to images of objects of $\X$, and write $\overline{\X}$ dually for the corresponding full subcategory of $\C/[\mathcal I]$. Thus these symbols denote stable and costable images, not differences of sets. Set
\[
 \underline{\Hom}_{\C}(A,B):=\C(A,B)/[\mathcal P](A,B),
\hspace{2.5mm}
 \overline{\Hom}_{\C}(A,B):=\C(A,B)/[\mathcal I](A,B).
\]
The pair $(\C,\C)$ is an $\X$-mutation pair if every $A\in\C$ occurs in an $\E$-triangle
\[
 A\longrightarrow X\longrightarrow B\dashrightarrow
\]
with $X\in\X$, the first morphism a left $\X$-approximation and the second a
right $\X$-approximation. Every $B\in\C$ must also occur in such a triangle.
This is the specialization of Zhou--Zhu \cite[Definition 3.1]{ZZtri}.

For the next theorem, the \emph{Auslander--Reiten duality hypotheses} are those of Zhou--Zhu \cite[Theorem 4.3]{ZZtri}: $\C$ is a Krull--Schmidt $k$-linear extriangulated category with Auslander--Reiten translations $\tau$ and $\tau^{-1}$, and there are functorial isomorphisms
\[
 \E(A,B)
 \cong D\overline{\Hom}_{\C}(B,\tau A)
 \cong D\underline{\Hom}_{\C}(\tau^{-1}B,A),
 \hspace{2.5mm} D=\Hom_k(-,k),
\]
for all $A,B\in\C$.

\begin{theorem}[Triangulated endpoint]\label{thm:known-triangulated}
Assume the Auslander--Reiten duality hypotheses just stated, and let $\X$ be strongly functorially finite. Then the following are equivalent:
\begin{enumerate}[{\rm(1)}]
\item $(\C,\C)$ is an $\X$-mutation pair.
\item the canonical pretriangulated category $\C/[\X]$ is triangulated.
\item $\tau\underline{\X}=\overline{\X}$.
\item the canonical suspension $\SigX$ is an autoequivalence of $\C/[\X]$.
\end{enumerate}
If $\C$ is triangulated, this specializes to
\[
        \C/[\X]\text{ triangulated}
        \quad\Longleftrightarrow\quad
        \tau\X=\X.
\]
\end{theorem}

\proof The equivalence of conditions \textup{(1)}--\textup{(3)} is due to Zhou--Zhu \cite[Theorem 4.3]{ZZtri}. The implication \textup{(2)}$\Rightarrow$\textup{(4)} is immediate, while \textup{(4)}$\Rightarrow$\textup{(2)} follows from Lemma~\ref{lem:stable-right} applied to the canonical right triangulated structure. The triangulated specialization of Zhou--Zhu \cite[Corollary 4.4]{ZZtri} recovers J\o rgensen \cite[Theorem 3.3]{Jor}. We use these facts only as a cited endpoint.  \qed

\subsection{
The canonical quotient suspension and elementary lemmas}\label{sec:canonical}

For the constructions involving $\SigX$, assume first only that $\X$ is strongly covariantly finite and use the canonical right triangulated quotient from Theorem~\ref{thm:known-onesided}. For each $A\in\C$, choose once and for all a left approximation $\E$-triangle
\[
A\xrightarrow{x_A}X_A\longrightarrow A^+\dashrightarrow,
       \hspace{2.5mm} X_A\in\X,
\]
so that
\[
        \SigX(\piX A)=\piX(A^+).
\]
The independence of the resulting functor, up to natural isomorphism, is part of the known one-sided quotient construction in Theorem~\ref{thm:known-onesided} and will not be reproved. When $\X$ is strongly functorially finite, this $\SigX$ agrees with the suspension in the canonical pretriangulated structure of Theorem~\ref{thm:known-pretri}.

We also need two elementary categorical observations.

\begin{lemma}\label{lem:objectwise-zero}
Let $F\colon\A\to\mathcal B$ be a functor and suppose that $\mathcal B$ has a zero object. Then the following are equivalent:
\begin{enumerate}[{\rm(1)}]
\item $F(A)$ is a zero object for every $A\in\A$.
\item $F\cong\zero$ naturally.
\end{enumerate}
If a distinguished zero object has been fixed and $F(A)=0$ literally for every $A$, then $F$ is literally the zero functor.
\end{lemma}

\begin{proof}
For every $A$, the unique morphism $F(A)\to0$ is an isomorphism because both objects are zero. These morphisms are automatically natural, again by uniqueness. Hence $F\cong\zero$. The final assertion follows because the only morphism $0\to0$ is the zero morphism.
\end{proof}

\begin{lemma}\label{lem:adjoint-zero}
Let $F$ and $G$ be additive endofunctors such that $F$ is left adjoint to $G$.
For every $n\geq 1$,
\[
F^n\cong\zero
\quad\Longleftrightarrow\quad
G^n\cong\zero.
\]
\end{lemma}

\begin{proof}
For every $n\geq 1$, the functor $F^n$ is left adjoint to $G^n$.
Suppose that $F^n\cong\zero$. Then, for all objects $A$ and $B$,
\[
\Hom(A,G^nB)\cong\Hom(F^nA,B)=0.
\]
Taking $A=G^nB$, we obtain $\id_{G^nB}=0$. Hence $G^nB$ is a zero
object for every $B$. Lemma~\ref{lem:objectwise-zero} then gives
$G^n\cong\zero$. The converse follows by the same argument.
\end{proof}

\begin{corollary}\label{cor:sigma-omega}
Assume that $\X$ is strongly functorially finite. For the canonical pretriangulated quotient,
\[
        \SigX^n\cong\zero
        \quad\Longleftrightarrow\quad
        \OmX^n\cong\zero.
\]
\end{corollary}

\begin{proof}
By Theorem~\ref{thm:known-pretri}, the relevant functors form an adjoint pair.
The result now follows immediately from Lemma~\ref{lem:adjoint-zero}.
\end{proof}

Starting from $A_0=A$, recursively choose the fixed left approximation $\E$-triangles
\begin{equation}\label{eq:tower}
 A_j\xrightarrow{x_j}X_j\longrightarrow A_{j+1}\dashrightarrow,
 \hspace{2.5mm} X_j\in\X,
  \hspace{2.5mm}0\leq j\leq n-1.
\end{equation}
By construction,
\begin{equation}\label{eq:sigma-tower}
        \SigX^j(\piX A)\cong\piX(A_j)
        \hspace{2.5mm}(0\leq j\leq n).
\end{equation}

Dually, when $\X$ is strongly contravariantly finite, start with $B_0=A$
and choose right approximation $\E$-triangles
\begin{equation}\label{eq:right-tower}
 B_{j+1}\longrightarrow X^j\xrightarrow{y_j}B_j\dashrightarrow,
  \hspace{2.5mm} X^j\in\X,
  \hspace{2.5mm} 0\leq j\leq n-1.
\end{equation}
Then
\begin{equation}\label{eq:omega-tower}
 \OmX^j(\piX A)\cong\piX(B_j)
  \hspace{2.5mm}(0\leq j\leq n).
\end{equation}

We will also use the following elementary property of ideal quotients.

\begin{lemma}\label{lem:zero-object-quotient}
Assume $\X=\add\X$. Then $\piX A$ is a zero object in $\C/[\X]$ if and only if $A\in\X$.
\end{lemma}

\begin{proof}
If $A\in\X$, then $\id_A$ factors through $\X$, so $\piX A$ is zero. Conversely, if $\piX A$ is zero, then $\id_A$ factors as $A\to X\to A$ with $X\in\X$. Hence $A$ is a direct summand of $X$ and belongs to $\X$.
\end{proof}

\begin{lemma}\label{lem:dimension-shifting}
Let $\C$ be a small $R$-linear extriangulated category. Assume that $\X$ is $n$-rigid, and let
\[
 A_j\xrightarrow{x_j}X_j\longrightarrow A_{j+1}\dashrightarrow,
 \hspace{2.5mm} X_j\in\X,
\]
be one of the left $\X$-approximation $\E$-triangles in \eqref{eq:tower}. Then, for every $X\in\X$, the following hold.
\begin{enumerate}[{\rm(1)}]
\item $\E(A_{j+1},X)=0$.
\item For $2\le r\le n$ there is a natural isomorphism
\[
        \E^r(A_{j+1},X)\cong \E^{r-1}(A_j,X).
\]
\item For $1\le r\le n-1$ there is a natural isomorphism
\[
        \E^r(X,A_{j+1})\cong \E^{r+1}(X,A_j).
\]
\end{enumerate}
Consequently,
\[
 \E^r(A_{j+1},X)=0
  \hspace{2mm}(0\le j\le n-1,\ 1\le r\le \min\{j+1,n\}),
\]
and, if $A_0\in\bigcap_{i=1}^n\X^{\perp_i}$, then
\[
        \E(X,A_j)=0 \hspace{2mm}(0\le j\le n-1).
\]
\end{lemma}

\begin{proof}
Fix $X\in\X$. Applying the long exact sequence of Gorsky--Nakaoka--Palu \cite[Theorem 3.6]{GNP} to the $\E$-triangle
\[
        A_j\xrightarrow{x_j}X_j\longrightarrow A_{j+1}\dashrightarrow
\]
gives the exact sequence
\[
 \C(X_j,X)\longrightarrow\C(A_j,X)
 \longrightarrow\E(A_{j+1},X)\longrightarrow\E(X_j,X).
\]
The first map is surjective because $x_j$ is a left $\X$-approximation, and the last term vanishes by rigidity. This proves (1).

For $2\le r\le n$, the same long exact sequence contains
\[
 \E^{r-1}(X_j,X)\longrightarrow \E^{r-1}(A_j,X)
 \longrightarrow \E^r(A_{j+1},X)
 \longrightarrow \E^r(X_j,X).
\]
Both outer terms vanish by $n$-rigidity, proving (2).

For $1\le r\le\min\{j+1,n\}$, iterate (2) through the preceding $\E$-triangles. The result terminates at the degree one vanishing in (1):
\[
 \E^r(A_{j+1},X)\cong\E(A_{j-r+2},X)=0.
\]

To prove \textup{(3)}, the long exact sequence contains, for $1\le r\le n-1$,
\[
 \E^r(X,X_j)\longrightarrow \E^r(X,A_{j+1})
 \longrightarrow \E^{r+1}(X,A_j)
 \longrightarrow \E^{r+1}(X,X_j).
\]
Again the outer terms vanish by $n$-rigidity, which proves (3).

If $A_0\in\bigcap_{i=1}^n\X^{\perp_i}$, then for $1\le j\le n-1$ repeated use of (3) gives
\[
        \E(X,A_j)\cong\E^{j+1}(X,A_0)=0,
\]
while the case $j=0$ is part of the assumption. Notice that the largest degree used here is $j+1\le n$ and vanishing in degree $n+1$ is not required.
\end{proof}

\begin{remark}
The small $R$-linear hypothesis is a convenient standard assumption ensuring that this calculus is available from Gorsky--Nakaoka--Palu \cite[Theorem 3.6]{GNP}.
\end{remark}

\section{Truncation and higher orthogonality}\label{sec:Truncation and higher orthogonality}

\subsection{One-sided truncation}
We establish the one-sided truncation mechanism underlying the recognition
theorem. Throughout this subsection, $(\C,\E,\mathfrak s)$ is a small
$R$-linear extriangulated category and $n\geq1$. Every subcategory denoted by
$\X=\add\X$ is closed under finite direct sums and direct summands. The
required approximation and $n$-rigidity assumptions are stated in each
result. We use the higher extension long exact sequences of
Gorsky--Nakaoka--Palu \cite[Theorem 3.6]{GNP}.

\begin{definition}\label{def:local-trunc}
For $m\geq0$, assume first that $\X$ is strongly covariantly finite and define
\[
 \operatorname{Tr}^{+}_{m}(\X)
 :=\{A\in\C\mid \SigX^{m}(\piX A)\cong0\}.
\]
Dually, when $\X$ is strongly contravariantly finite, define
\[
 \operatorname{Tr}^{-}_{m}(\X)
 :=\{A\in\C\mid \OmX^{m}(\piX A)\cong0\}.
\]
Here the zeroth power is the identity.  If $\X=\add\X$, then Lemma~\ref{lem:zero-object-quotient} gives
$\operatorname{Tr}^{+}_{0}(\X)=\operatorname{Tr}^{-}_{0}(\X)=\X$.
Since the canonical suspension and loop are additive, these loci are closed under finite direct sums and direct summands, and
\[
 \operatorname{Tr}^{\pm}_{m}(\X)\subseteq
 \operatorname{Tr}^{\pm}_{m+1}(\X)
 \hspace{2.5mm}(m\geq0).
\]
\end{definition}

\begin{proposition}\label{prop:left-to-sigma}
Let $\X=\add\X$ be strongly covariantly finite and $n$-rigid. If
\[
        \X=\bigcap_{i=1}^n{}^{\perp_i}\X,
\]
then
\[
        \SigX^n\cong\zero.
\]
\end{proposition}

\begin{proof}
Let $A\in\C$, set $A_0=A$, and consider the $\E$-triangle \eqref{eq:tower}. By Lemma~\ref{lem:dimension-shifting}, for every $X\in\X$,
\[
        \E^r(A_n,X)=0,
        \qquad 1\le r\le n.
\]
Hence
\[
        A_n\in\bigcap_{r=1}^n{}^{\perp_r}\X=\X.
\]
Equation \eqref{eq:sigma-tower} therefore gives
\[
        \SigX^n(\piX A)\cong\piX(A_n)\cong0
\]
for every $A\in\C$. By Lemma~\ref{lem:objectwise-zero}, $\SigX^n\cong\zero$.
\end{proof}

\begin{theorem}\label{thm:local-reconstruction}
Let $\X=\add\X$ be strongly covariantly finite and $n$-rigid. Then
\[
 \X
 =
 \operatorname{Tr}^{+}_{n}(\X)
 \cap
 \bigcap_{i=1}^{n}\X^{\perp_i}.
\]
Equivalently, for every $A\in\bigcap_{i=1}^{n}\X^{\perp_i}$,
\[
 \SigX^n(\piX A)\cong0
 \quad\Longleftrightarrow\quad
 A\in\X.
\]
\end{theorem}

\begin{proof}
The inclusion from left to right follows from $n$-rigidity and the fact that every object of $\X$ is zero in $\C/[\X]$.

For the reverse inclusion, let
\[
 A=A_0\in\operatorname{Tr}^{+}_{n}(\X)
       \cap\bigcap_{i=1}^{n}\X^{\perp_i}
\]
and consider the $\E$-triangle \eqref{eq:tower}.  By \eqref{eq:sigma-tower}, the local truncation condition gives
\[
 \piX(A_n)\cong\SigX^n(\piX A)\cong0.
\]
Hence $A_n\in\X$ by Lemma~\ref{lem:zero-object-quotient}.  On the other hand, Lemma~\ref{lem:dimension-shifting} gives
\[
 \E(X,A_j)=0
  \hspace{2.5mm}
 (X\in\X,\ 0\leq j\leq n-1).
\]
We show by descending induction that $A_j\in\X$ for every $j$.  If $A_{j+1}\in\X$, then the extension class of
\[
 A_j\longrightarrow X_j\longrightarrow A_{j+1}\dashrightarrow
\]
lies in $\E(A_{j+1},A_j)=0$.  The $\E$-triangle therefore splits, so $A_j$ is a direct summand of $X_j\in\X$ and hence belongs to $\X$.  Starting from $A_n\in\X$ yields $A=A_0\in\X$.
\end{proof}

\begin{corollary}[Global truncation forces right maximal orthogonality]\label{prop:sigma-to-right}
Let $\X=\add\X$ be strongly covariantly finite and $n$-rigid. If
\[
 \SigX^n\cong\zero,
\]
then
\[
 \X=\bigcap_{i=1}^n\X^{\perp_i}.
\]
\end{corollary}

\begin{proof}
Global truncation means that
\[
\operatorname{Tr}^{+}_{n}(\X)=\C.
\]
Hence every object of $\C$ satisfies the local truncation condition appearing
in Theorem~\ref{thm:local-reconstruction}. The conclusion therefore follows
directly from that theorem.
\end{proof}

\begin{theorem}\label{thm:local-reconstruction-dual}
Let $\X=\add\X$ be strongly contravariantly finite and $n$-rigid. Then
\[
 \X
 =
 \operatorname{Tr}^{-}_{n}(\X)
 \cap
 \bigcap_{i=1}^{n}{}^{\perp_i}\X.
\]
Equivalently, for every $A\in\bigcap_{i=1}^{n}{}^{\perp_i}\X$,
\[
 \OmX^n(\piX A)\cong0
 \quad\Longleftrightarrow\quad
 A\in\X.
\]
\end{theorem}

\begin{proof}
The inclusion from left to right follows from $n$-rigidity and the fact that
objects of $\X$ are zero in the quotient. Conversely, let
\[
 A=B_0\in\operatorname{Tr}^{-}_{n}(\X)
       \cap\bigcap_{i=1}^{n}{}^{\perp_i}\X
\]
and use the right approximation $\E$-triangle \eqref{eq:right-tower}. Equation
\eqref{eq:omega-tower} and Lemma~\ref{lem:zero-object-quotient} give
$B_n\in\X$.

Fix $X\in\X$. The long exact sequences of Gorsky--Nakaoka--Palu
\cite[Theorem 3.6]{GNP}, applied to \eqref{eq:right-tower}, give
\[
 \E(B_j,X)\cong\E^{j+1}(B_0,X)=0
 \hspace{2.5mm}(0\leq j\leq n-1).
\]
Suppose inductively that $B_{j+1}\in\X$. The extension class of
\[
 B_{j+1}\longrightarrow X^j\longrightarrow B_j\dashrightarrow
\]
belongs to $\E(B_j,B_{j+1})=0$. Hence the triangle splits and $B_j$ is a
direct summand of $X^j$. Since $\X=\add\X$, it follows that $B_j\in\X$.
Descending induction from $B_n\in\X$ yields $A=B_0\in\X$.
\end{proof}

\begin{proposition}\label{prop:dual}
Let $\X=\add\X$ be strongly contravariantly finite and $n$-rigid.
\begin{enumerate}[{\rm(1)}]
\item If $\displaystyle \X=\bigcap_{i=1}^n\X^{\perp_i}$, then $\OmX^n\cong\zero$.
\item If $\OmX^n\cong\zero$, then $\displaystyle \X=\bigcap_{i=1}^n{}^{\perp_i}\X$.
\end{enumerate}
\end{proposition}

\begin{proof}
The first assertion follows by the dual argument of
Proposition~\ref{prop:left-to-sigma}. For the second assertion, global
truncation means that
$
\operatorname{Tr}^{-}_{n}(\X)=\C.
$
Thus every object of $\C$ satisfies the local truncation condition required in
Theorem~\ref{thm:local-reconstruction-dual}. Applying that theorem to each
object of $\C$ gives the desired conclusion.
\end{proof}

\begin{theorem}\label{thm:one-sided}
Let $(\C,\E,\mathfrak s)$ be a small $R$-linear extriangulated category.
\begin{enumerate}[{\rm(1)}]
\item If $\X=\add\X$ is strongly covariantly finite and $n$-rigid, then
\[
 \X=\bigcap_{i=1}^{n}{}^{\perp_i}\X
 \quad\Longrightarrow\quad
 \SigX^n\cong\zero
 \quad\Longrightarrow\quad
 \X=\bigcap_{i=1}^{n}\X^{\perp_i}.
\]
\item If $\X=\add\X$ is strongly contravariantly finite and $n$-rigid, then
\[
 \X=\bigcap_{i=1}^{n}\X^{\perp_i}
 \quad\Longrightarrow\quad
 \OmX^n\cong\zero
 \quad\Longrightarrow\quad
 \X=\bigcap_{i=1}^{n}{}^{\perp_i}\X.
\]
\end{enumerate}
\end{theorem}

\begin{proof}
Part \textup{(1)} is Propositions~\ref{prop:left-to-sigma} and~\ref{prop:sigma-to-right}. Part \textup{(2)} is Proposition~\ref{prop:dual}.
\end{proof}

\begin{remark}\label{rem:crossed}
The global crossed implication
\[
 \X=\bigcap_{i=1}^n{}^{\perp_i}\X
 \Longrightarrow
 \SigX^n\cong0
 \Longrightarrow
 \X=\bigcap_{i=1}^n\X^{\perp_i}
\]
is the categorical shadow of the stronger local statement in
Theorem~\ref{thm:local-reconstruction}. On the locus defined by right higher
orthogonality, the $n$th canonical suspension detects precisely the objects
already killed by the quotient. The loop functor gives the dual detection
statement on the locus defined by left higher orthogonality.
\end{remark}

\subsection{Recognition of higher maximal orthogonality}\label{sec:recognition}

\begin{theorem}\label{thm:recognition}
Let $(\C,\E,\mathfrak s)$ be a small $R$-linear extriangulated category and $\X=\add\X$ be a strongly functorially finite $n$-rigid subcategory. Then the following are equivalent:
\begin{enumerate}[{\rm(1)}]
\item $\X$ satisfies the two-sided maximal $n$-orthogonality condition of {\rm Definition~\ref{def:maxorth}}.
\item $\SigX^n\cong\zero$ on $\C/[\X]$.
\item $\OmX^n\cong\zero$ on $\C/[\X]$.
\end{enumerate}
\end{theorem}

\begin{proof}
Assume (1). Then $\X=\bigcap_{i=1}^n{}^{\perp_i}\X$, and Proposition~\ref{prop:left-to-sigma} yields (2). The equivalence of (2) and (3) follows from Corollary~\ref{cor:sigma-omega}.

Conversely, assume (2). By Corollary~\ref{prop:sigma-to-right},
\[
        \X=\bigcap_{i=1}^n\X^{\perp_i}.
\]
By Corollary~\ref{cor:sigma-omega}, condition (3) holds as well. Hence Proposition~\ref{prop:dual}(2) yields
\[
        \X=\bigcap_{i=1}^n{}^{\perp_i}\X.
\]
The two equalities together are precisely condition \textup{(1)}.
\end{proof}

\begin{remark}\label{rem:hypotheses}
The following observations separate the rigidity and approximation assumptions in Theorem~\ref{thm:recognition}.
\begin{enumerate}[{\rm(1)}]
	\item The $n$-rigidity hypothesis cannot be omitted.
	
	Let $\Lambda=k[\varepsilon]/(\varepsilon^2)$ and $\C=\mathrm{mod}\text{-}\Lambda$ with its usual exact structure. Put $\X=\C$. Then $\C/[\X]=0$, so every power of the canonical suspension and loop is zero. However $\X$ is not rigid. For the simple module $S=\Lambda/(\varepsilon)$, the following nonsplit sequence
		\[
		0\longrightarrow S\longrightarrow\Lambda\longrightarrow S\longrightarrow0
		\]
		gives $\operatorname{Ext}^1_\Lambda(S,S)\neq0$. This is a nonsplit exact
		example that is extriangulated but not triangulated. It shows that
		nilpotency of the canonical suspension and loop alone cannot recognize cluster tilting.
	\item The two approximation directions are independent data.
		
	Truncation does not by itself recover missing approximation theory. Approximation finiteness is an independent part of higher cluster tilting theory in infinite type. Holm--J\o rgensen \cite[Theorems D and E]{HJ} distinguish weakly $d$ cluster tilting subcategories from $d$ cluster tilting subcategories by separate left and right approximation conditions. In their $A_\infty$ model, weakly $d$ cluster tilting subcategories need not be functorially finite. Holm--J\o rgensen \cite[Propositions 3.4 and 3.5]{HJ} characterize the two approximation properties separately. Thus maximal higher orthogonality by itself does not supply approximation finiteness.
		
	Accordingly, Theorem~\ref{thm:one-sided} gives the natural one-sided statement. Under strong covariant finiteness, the condition $\SigX^n\cong\zero$ transfers maximal orthogonality from the left to the right. The dual statement transfers it in the opposite direction. The proof of the two-sided recognition theorem uses both approximation directions. This observation explains the formulation of the theorem. It is not asserted here as a counterexample showing that no weaker two-sided hypothesis can exist.
\end{enumerate}	
\end{remark}

\begin{corollary}\label{cor:classical}
Assume in addition that $\C$ has enough projectives and enough injectives. Then, for a strongly functorially finite $n$-rigid subcategory $\X=\add\X$, the following are equivalent.
\begin{enumerate}[{\rm(1)}]
	\item
	$\X\text{ is }(n+1)\text{-cluster tilting}$.
	\item $(\X,\X)$ is an $n$-cotorsion pair in the sense of He--Zhou \cite[Definition 2.3]{HZ}.
	\item $\SigX^n\cong\zero$.
	\item $\OmX^n\cong\zero$.
\end{enumerate}
\end{corollary}

\proof This follows from Theorem~\ref{thm:recognition}, the characterization of He--Zhou \cite[Definition 2.3 and Theorem 3.5]{HZ}, and the comparison theorem of Gorsky--Nakaoka--Palu \cite[Corollary 3.23]{GNP}.  \qed

\begin{corollary}\label{cor:triangulated-ntruncated}
Let $\T$ be a triangulated category and $\X=\add\X$ be a functorially finite subcategory satisfying
\[
 \T(X,X'[i])=0
\hspace{2.5mm}
 (X,X'\in\X,\ 1\leq i\leq n).
\]
Equip $\T/[\X]$ with its canonical pretriangulated structure.  Then the following are equivalent:
\begin{enumerate}[{\rm(1)}]
\item $\X$ is an $(n+1)$ cluster tilting subcategory of $\T$.
\item $\T/[\X]$ is $n$-truncated in the sense of Mochizuki--Nakaoka--Ogawa \cite[Definition 1.2.3]{MNO}.
\item $\SigX^n\cong\zero$.
\item $\OmX^n\cong\zero$.
\end{enumerate}
Moreover, for every $A\in\T$ one has the objectwise reconstruction formulas
\[
 A\in\X
 \quad\Longleftrightarrow\quad
 \left\{
 \begin{array}{l}
 \T(\X,A[i])=0\quad(1\leq i\leq n),\\
 \SigX^n(\piX A)\cong0,
 \end{array}
 \right.
\]
and
\[
 A\in\X
 \quad\Longleftrightarrow\quad
 \left\{
 \begin{array}{l}
 \T(A,\X[i])=0\quad(1\leq i\leq n),\\
 \OmX^n(\piX A)\cong0.
 \end{array}
 \right.
\]
\end{corollary}

\begin{proof}
The higher extensions are the shifted Hom groups, and the standard long exact Hom sequences give exactly the dimension-shifting argument used in Theorems~\ref{thm:local-reconstruction}, \ref{thm:local-reconstruction-dual} and~\ref{thm:recognition}. Therefore no smallness assumption is needed. Zhou--Zhu \cite[Remark 3.20]{ZZtri} show that ordinary functorial finiteness agrees with strong functorial finiteness in a triangulated category. Hence \textup{(1)}, \textup{(3)}, and \textup{(4)} are equivalent. Mochizuki--Nakaoka--Ogawa \cite[Definition 1.2.3]{MNO} define $n$-truncatedness by the vanishing of the $n$th suspension, equivalently the $n$th loop. Thus \textup{(2)} is equivalent to \textup{(3)} and \textup{(4)}. The final two formulas are the triangulated specializations of Theorems~\ref{thm:local-reconstruction} and~\ref{thm:local-reconstruction-dual}.
\end{proof}

\begin{remark}\label{rem:MNO-converse}
Mochizuki--Nakaoka--Ogawa
\cite[Example 2.1.5(2), Definition 2.3.1 and Proposition 2.3.23]{MNO}
specialize their heart construction to the ideal quotient $\T/[\X]$ by an
$(n+1)$ cluster tilting subcategory and prove that the resulting
pretriangulated category is $n$-truncated. In this specialization, the suspension functor in their Definition 2.3.1 is
induced by the same left $\X$-approximation triangles as the canonical
quotient suspension used here. Hence the two suspension functors agree up to
natural isomorphism. The implication
\textup{(2)}$\Rightarrow$\textup{(1)} in
Corollary~\ref{cor:triangulated-ntruncated} runs in the opposite direction.
For the canonical ideal quotient by a functorially finite $n$-rigid
subcategory, $n$-truncatedness forces higher cluster tilting. The objectwise
formulas in Theorems~\ref{thm:local-reconstruction} and~\ref{thm:local-reconstruction-dual} are stronger because they require truncation only at the object
under consideration.
\end{remark}

We give two elementary computations that separate local from global truncation
and show that the exponent in the main theorem is optimal. Let $k$ be a field.
For $N\geq2$, let $\mathcal S_N$ be the semisimple $k$-linear category with
indecomposable objects $S_0,\ldots,S_{N-1}$, with indices read modulo $N$,
and
\[
 \Hom(S_i,S_j)=
 \begin{cases}
 k,&i=j,\\
 0,&i\ne j.
 \end{cases}
\]
Give $\mathcal S_N$ the split triangulated structure whose suspension satisfies
$S_i[1]=S_{i+1}$. This is the standard split triangulated structure described
in the triangulated formalism of Beligiannis--Reiten \cite[Chapter I]{BR}.
Explicitly, its distinguished triangles are the triangles isomorphic to
finite direct sums of rotations of
\[
 X\xrightarrow{1_X}X\longrightarrow0\longrightarrow X[1]
 \hspace{2.5mm}\text{and}\hspace{2.5mm}
 X\longrightarrow0\longrightarrow X[1]\xrightarrow{1_{X[1]}}X[1].
\]

\begin{example}[Local truncation need not be global]\label{ex:local-not-global}
Fix $n\geq1$, take $N=n+2$, and put $\X=\add(S_0)$ in $\mathcal S_N$.
The subcategory $\X$ is functorially finite and $n$-rigid. For $j\ne0$, the
zero map $S_j\to0$ is a left $\X$-approximation and its triangle has third
term $S_{j+1}$. Therefore
\[
 \SigX^r(\piX S_j)\cong\piX(S_{j+r})
\]
until the orbit reaches $S_0$. In particular,
\[
 \SigX^n(\piX S_2)=0,
\hspace{2.5mm}
 \SigX^n(\piX S_1)=\piX(S_{n+1})\ne0.
\]
Thus $\operatorname{Tr}^{+}_{n}(\X)$ is a proper, nontrivial local
truncation locus while $\SigX^n$ is not the zero functor. Moreover,
\[
 \E^n(S_0,S_2)=\Hom(S_0,S_2[n])\cong k,
\]
so $S_2$ fails the orthogonality hypothesis in
Theorem~\ref{thm:local-reconstruction}. This also shows directly why local
truncation alone cannot recover $\X$.
\end{example}

\begin{example}\label{ex:sharp}
Fix $n\geq1$, take $N=n+1$, and put $\X=\add(S_0)$ in $\mathcal S_N$.
For every $1\leq j\leq n$,
\[
 \E^{N-j}(S_0,S_j)\cong k,
 \hspace{2.5mm}
 \E^{j}(S_j,S_0)\cong k.
\]
It follows that $\X$ is two-sided maximally $n$-orthogonal. The same direct
calculation of the quotient suspension gives
\[
 \SigX^{n-1}(\piX S_1)=\piX(S_n)\ne0,
\hspace{2.5mm}
 \SigX^n(\piX S_1)=\piX(S_0)=0.
\]
Hence $\nu_{\X}(\mathcal S_N)=n$, and
\[
 \E^{n+1}(S_0,S_0)=\Hom(S_0,S_0[n+1])\cong k.
\]
This realizes both sharp conclusions of Theorem~\ref{thm:sharp-truncation}.
\end{example}

In exact categories, the canonical suspension and loop admit a concrete description through known approximation resolutions.

\begin{proposition}\label{prop:exact-two-resolutions}
	Let $\mathcal E$ be an exact category and $\X=\add\X$ satisfy
	\[
	\operatorname{Ext}^i_{\mathcal E}(\X,\X)=0,
	\hspace{2.5mm} 1\leq i\leq n.
	\]
	Then the following are equivalent.
	\begin{enumerate}[{\rm(1)}]
		\item $\X$ is an $(n+1)$ cluster tilting subcategory of $\mathcal E$.
		\item Every $E\in\mathcal E$ admits admissible exact sequences
		\[
		0\longrightarrow E\longrightarrow X^0\longrightarrow X^1\longrightarrow\cdots\longrightarrow X^n\longrightarrow0
		\]
		and
		\[
		0\longrightarrow X_n\longrightarrow\cdots\longrightarrow X_1\longrightarrow X_0\longrightarrow E\longrightarrow0
		\]
		with all terms $X^i,X_i$ in $\X$.
	\end{enumerate}
	Here an admissible exact sequence means a sequence obtained by splicing
	conflations in the exact structure of $\mathcal E$.
	More precisely, the first family of sequences exists for every $E$ if and
	only if $\X$ is strongly covariantly finite and $\SigX^n\cong\zero$.
	Dually, the second family exists for every $E$ if and only if $\X$ is
	strongly contravariantly finite and $\OmX^n\cong\zero$.
\end{proposition}

\begin{proof}
	The equivalence of \textup{(1)} and \textup{(2)} is the criterion of
	Kvamme \cite[Proposition 4.4]{Kvamme}, applied with $d=n+1$. We do not reprove that
	result. It remains to verify the two one-sided assertions.
	
	Consider the first admissible exact sequence and write it as
	\[
	0\longrightarrow A_0=E\longrightarrow X^0\longrightarrow X^1
	\longrightarrow\cdots\longrightarrow X^n\longrightarrow0.
	\]
	For $0\le j\le n-1$, let $A_{j+1}$ be the cokernel of the admissible monomorphism $A_j\to X^j$. Exactness gives conflations
	\[
	0\longrightarrow A_j\longrightarrow X^j\longrightarrow A_{j+1}\longrightarrow0,
	\hspace{2.5mm} 0\le j\le n-1,
	\]
	with $A_n\cong X^n\in\X$. We claim that each $A_j\to X^j$ is a left $\X$-approximation. Fix $X\in\X$. Dimension shifting along the tail of the resolution, together with the $n$-rigidity of $\X$, gives
	\[
	\operatorname{Ext}^1_{\mathcal E}(A_{j+1},X)
	\cong
	\operatorname{Ext}^{\,n-j}_{\mathcal E}(X^n,X)=0.
	\]
	Applying $\operatorname{Hom}_{\mathcal E}(-,X)$ to
	$0\to A_j\to X^j\to A_{j+1}\to0$ therefore shows that
	\[
	\operatorname{Hom}_{\mathcal E}(X^j,X)
	\longrightarrow
	\operatorname{Hom}_{\mathcal E}(A_j,X)
	\]
	is surjective. Thus the $n$ conflations above form a valid tower of left approximations for the canonical one-sided suspension, and
	\[
	\SigX^n(\piX E)\cong\piX(A_n)=0.
	\]
	Since this holds for every $E$, Lemma~\ref{lem:objectwise-zero} gives $\SigX^n\cong\zero$.
	The same argument also proves strong covariant finiteness, because the first
	conflation supplies a left approximation inflation for every $E$.
	
	Conversely, if the canonical one-sided suspension is defined and $\SigX^n\cong\zero$, its defining tower consists of $n$ conflations
	\[
	A_j\longrightarrow X^j\longrightarrow A_{j+1},
	\hspace{2.5mm} 0\le j\le n-1.
	\]
	By Lemma~\ref{lem:zero-object-quotient}, the terminal object $A_n$ lies in $\X$. Splicing these conflations gives the admissible exact sequence
	\[
	0\longrightarrow E\longrightarrow X^0\longrightarrow\cdots
	\longrightarrow X^{n-1}\longrightarrow A_n\longrightarrow0,
	\]
	which is the first resolution after renaming $A_n$ as $X^n$. This proves
	the first one-sided equivalence. The statement for $\OmX^n$ and the second
	resolution is dual.
\end{proof}

\subsection{Nilpotency index of the canonical quotient suspension}

\begin{definition}\label{def:nil-index}
Assume that $\X$ is strongly functorially finite.  Define
\[
 \nu_{\X}(\C)
 :=\min\{m\geq1\mid \SigX^m\cong\zero\}
 \in\mathbb N\cup\{\infty\}.
\]
Here we use the convention $\min\varnothing=\infty$.
By Corollary~\ref{cor:sigma-omega}, the same number is obtained by replacing $\SigX$ with $\OmX$.
\end{definition}

\begin{proposition}\label{prop:rigidity-collapse}
Let $(\C,\E,\mathfrak s)$ be a small $R$-linear extriangulated category and $\X=\add\X$.
\begin{enumerate}[{\rm(1)}]
\item Suppose that $\X$ is strongly contravariantly finite,
\[
 \X=\bigcap_{i=1}^{n}\X^{\perp_i},
\]
and
\[
 \E^{n+1}(\X,\X)=0.
\]
Then $\C=\X$ and $\E=0$.
\item Dually, the same conclusion holds if $\X$ is strongly covariantly finite,
\[
 \X=\bigcap_{i=1}^{n}{}^{\perp_i}\X,
\]
and $\E^{n+1}(\X,\X)=0$.
\end{enumerate}
\end{proposition}

\begin{proof}
We prove \textup{(1)}.  The displayed maximal orthogonality implies
$\E^i(\X,\X)=0$ for $1\leq i\leq n$. Combined with the last hypothesis, this yields that \(\mathcal{X}\) is \((n+1)\)-rigid. By the first assertion of Proposition~\ref{prop:dual},
$\OmX^n\cong\zero$.

Fix $A\in\C$, put $B_0=A$, and choose an iterated right $\X$-approximation $\E$-triangle
\[
 B_{j+1}\longrightarrow X_j\longrightarrow B_j\dashrightarrow,
\hspace{2.5mm} X_j\in\X,
\hspace{2.5mm} 0\leq j\leq n-1.
\]
The canonical loop construction gives
$\OmX^n(\piX A)\cong\piX(B_n)$. Hence
$B_n\in\X$ by Lemma~\ref{lem:zero-object-quotient}. Fix $X\in\X$. For each $0\leq j\leq n-1$, the long exact sequence of Gorsky--Nakaoka--Palu \cite[Theorem 3.6]{GNP} attached to
\[
 B_{j+1}\longrightarrow X_j\longrightarrow B_j\dashrightarrow
\]
contains
\[
 \E^{j+1}(X,X_j)\longrightarrow
 \E^{j+1}(X,B_j)\longrightarrow
 \E^{j+2}(X,B_{j+1})\longrightarrow
 \E^{j+2}(X,X_j).
\]
Since $1\leq j+1\leq n$ and $2\leq j+2\leq n+1$, both outer terms vanish by $(n+1)$-rigidity.  Hence
\[
 \E^{j+1}(X,B_j)\cong\E^{j+2}(X,B_{j+1})
 \hspace{2.5mm}(0\leq j\leq n-1).
\]
Iterating these isomorphisms gives
\[
 \E(X,A)
 \cong\E^2(X,B_1)
 \cong\cdots\cong
 \E^{n+1}(X,B_n)=0,
\]
where the last term vanishes because $B_n\in\X$.
Thus every object of $\X$ is projective in $\C$.

By the construction of positive extensions in Gorsky--Nakaoka--Palu
\cite[Section 3]{GNP}, $\E^m=\E^{\diamond m}$. Thus projectivity of $X\in\X$ implies
$\E^m(X,-)=0$ for every $m\geq1$.  Hence every $A\in\C$ belongs to
$\bigcap_{i=1}^{n}\X^{\perp_i}=\X$, so $\C=\X$.  Finally
$\E(\X,\X)=0$, and therefore $\E=0$ on $\C$.  Part \textup{(2)} is dual.
\end{proof}

\begin{remark}\label{rem:HMP-collapse}
In an abelian category, Huerta--Mendoza--P\'erez
\cite[Proposition 5.27]{HMP} proved a related collapse statement for an
$(n+1)$ cluster tilting subcategory. They showed that one further self
extension vanishing forces the cluster tilting subcategory to coincide with
the projective objects and with the injective objects.
Proposition~\ref{prop:rigidity-collapse} has a one-sided formulation in a
general small $R$-linear extriangulated category. Its role here is to control
the \emph{least} truncation exponent of the canonical ideal quotient.
\end{remark}

\begin{theorem}\label{thm:sharp-truncation}
Let $(\C,\E,\mathfrak s)$ be a small $R$-linear extriangulated category with $\E\neq0$, and let $\X=\add\X$ be strongly functorially finite.  Assume that $\X$ satisfies the two-sided maximal $n$-orthogonality condition.  Then
\[
 \E^{n+1}(\X,\X)\neq0
\]
and
\[
 \nu_{\X}(\C)=n.
\]
In particular, the positive global dimension of $\C$ in the sense of Gorsky--Nakaoka--Palu \cite[Definition 3.30]{GNP} is at least $n+1$ (or infinite).
\end{theorem}

\begin{proof}
If $\E^{n+1}(\X,\X)=0$, either part of Proposition~\ref{prop:rigidity-collapse} would imply $\E=0$, contrary to the hypothesis.  This proves the first assertion.

By Theorem~\ref{thm:recognition}, $\SigX^n\cong\zero$, so
$\nu_{\X}(\C)\leq n$.  For $n=1$ this already gives
$\nu_{\X}(\C)=1$.  Assume $n\geq2$ and suppose that
$\nu_{\X}(\C)=m<n$.  Since maximal $n$-orthogonality implies $n$-rigidity, $\X$ is $m$-rigid. Applying Theorem~\ref{thm:recognition} with $m$ in place of $n$ shows that $\X$ is two-sided maximally $m$-orthogonal.  The original $n$-rigidity also gives
$\E^{m+1}(\X,\X)=0$.  Proposition~\ref{prop:rigidity-collapse}, now at level $m$, forces $\E=0$, again a contradiction.  Hence $\nu_{\X}(\C)=n$.

Finally, $\E^{n+1}(\X,\X)\neq0$ means that the bifunctor $\E^{n+1}$ is nonzero, so the positive global dimension cannot be at most $n$.
\end{proof}

\begin{corollary}\label{cor:strict-triangulated}
Let $\T\neq0$ be a triangulated category and $\X=\add\X$ be a functorially finite $(n+1)$ cluster tilting subcategory.  Then the canonical suspension of $\T/[\X]$ has nilpotency index exactly $n$.
\[
 \min\{m\geq1\mid \Sigma_{\X}^{m}\cong\zero\}=n.
\]
Equivalently, the canonical quotient is $n$-truncated and, for $n\geq2$, it is not $(n-1)$-truncated.
\end{corollary}

\begin{proof}
The proof of Proposition~\ref{prop:rigidity-collapse} and Theorem~\ref{thm:sharp-truncation} uses only approximation towers and the long exact sequences for higher extensions.  In a triangulated category these are the ordinary long exact Hom sequences with
$\E^i(A,B)=\T(A,B[i])$, so the same collapse argument applies without any smallness assumption.  Moreover, the associated extriangulated structure of a nonzero triangulated category is not split. Indeed, if
$\T(A,B[1])=0$ for all $A,B$, then essential surjectivity of the shift gives $\T(A,C)=0$ for all $A,C$, forcing $\T=0$. Hence the contradiction argument in Theorem~\ref{thm:sharp-truncation} yields the stated exact nilpotency index. The equivalence with $n$-truncatedness follows from Mochizuki--Nakaoka--Ogawa \cite[Definition 1.2.3]{MNO}.
\end{proof}

\begin{corollary}\label{cor:pgldim-obstruction}
Let $(\C,\E,\mathfrak s)$ be a nonsplit small $R$-linear extriangulated category of finite positive global dimension $d$.  If $\C$ admits a strongly functorially finite two-sided maximally $n$-orthogonal subcategory, then
\[
 n\leq d-1.
\]
Equivalently, no such subcategory can exist at level $n\geq d$.
\end{corollary}

\begin{proof}
This is immediate from Theorem~\ref{thm:sharp-truncation}, which gives
$\E^{n+1}\neq0$ and hence $d\geq n+1$.
\end{proof}

\begin{corollary}\label{cor:minimal-level}
Let $(\C,\E,\mathfrak s)$ be a small $R$-linear extriangulated category and let
$\X=\add\X$ be strongly functorially finite and $n$-rigid for some $n\geq1$.
For $1\leq m\leq n$, the following are equivalent:
\begin{enumerate}[{\rm(1)}]
\item $\nu_{\X}(\C)=m$.
\item $\X$ satisfies the two-sided maximal $m$-orthogonality condition, and it satisfies no two-sided maximal $r$-orthogonality condition for $1\leq r<m$.
\end{enumerate}
If $\C$ has enough projectives and injectives, this is equivalently the least $m$ for which
$\X$ is an $(m+1)$ cluster tilting subcategory.
\end{corollary}

\begin{proof}
Assume first that $\nu_{\X}(\C)=m$.  Since $m\leq n$, the subcategory $\X$ is $m$-rigid, and Theorem~\ref{thm:recognition} applied with $n=m$ shows that $\X$ satisfies two-sided maximal $m$-orthogonality.  If it satisfied the corresponding condition for some $r<m$, then another application of Theorem~\ref{thm:recognition} would give $\SigX^r\cong\zero$, contrary to the minimality of $m$.

Conversely, suppose that the condition in \textup{(2)} holds. Theorem~\ref{thm:recognition}, with $n=m$, gives $\SigX^m\cong\zero$, so $\nu_{\X}(\C)\leq m$. If $\nu_{\X}(\C)=r<m$, then $n$-rigidity implies $r$-rigidity and Theorem~\ref{thm:recognition}, now with $n=r$, would force two-sided maximal $r$-orthogonality, a contradiction. Hence $\nu_{\X}(\C)=m$. The final assertion follows from Corollary~\ref{cor:classical}.
\end{proof}

\begin{corollary}\label{cor:nil-index-invariant}
Let $(\C,\X)$ and $(\mathcal D,\mathcal Y)$ be pairs for which the canonical pretriangulated ideal quotients are defined.  Suppose
\[
 \Phi\colon \C/[\X]\xrightarrow{\sim}\mathcal D/[\mathcal Y]
\]
is an equivalence compatible with the canonical suspensions, in the sense that
$\Phi\SigX\cong\Sigma_{\mathcal Y}\Phi$ naturally.  Then
\[
 \nu_{\X}(\C)=\nu_{\mathcal Y}(\mathcal D).
\]
In particular, under the hypotheses of Corollary~\ref{cor:minimal-level} on both pairs, an equivalence of the ideal quotients that is compatible with the canonical suspensions preserves the minimal higher maximal orthogonality level. In the classical higher extension setting it preserves the minimal cluster tilting level.
\end{corollary}

\begin{proof}
For every $m\geq1$, iteration of the compatibility isomorphism gives
$\Phi\SigX^m\cong\Sigma_{\mathcal Y}^m\Phi$.  Since $\Phi$ is an equivalence,
$\SigX^m\cong\zero$ if and only if $\Sigma_{\mathcal Y}^m\cong\zero$.
Thus the two sets of nilpotency exponents coincide. In particular, either both are infinite or they have the same least element.  This proves the first assertion, and the remaining statements follow from Corollary~\ref{cor:minimal-level}.
\end{proof}

The following proposition shows that an equivalence of canonical pretriangulated quotients does not, by itself, preserve the rigidity of the ambient subcategory.

\begin{proposition}\label{prop:no-internal-rigidity}
	Fix $n\geq1$. There exist two exact, hence extriangulated, pairs $(\C_0,\X_0)$ and $(\C_1,\X_1)$ such that their canonical pretriangulated ideal quotients are equivalent to the zero category, while $\X_0$ satisfies the two-sided maximal $n$-orthogonality condition (and is standard $(n+1)$ cluster tilting in this exact setting), whereas $\X_1$ is not even rigid.
\end{proposition}

\begin{proof}
	Let $k$ be a field and take $\C_0=\mathrm{mod}\text{-}k$ with its usual exact structure and $\X_0=\C_0$. Since $\C_0$ is semisimple, all positive higher extensions vanish. Hence
	\[
	\X_0
	=\bigcap_{i=1}^{n}{}^{\perp_i}\X_0
	=\bigcap_{i=1}^{n}\X_0^{\perp_i},
	\]
	and $\X_0$ is strongly functorially finite. Thus $\X_0$ satisfies the two-sided maximal $n$-orthogonality condition and, in this exact setting, is a standard $(n+1)$ cluster tilting subcategory.
	
	Now let $\Lambda=k[\varepsilon]/(\varepsilon^2)$, put $\C_1=\mathrm{mod}\text{-}\Lambda$, and take $\X_1=\C_1$. Again $\X_1$ is strongly functorially finite. In both cases every identity morphism factors through an object of $\X_i$, so
	\[
	\C_0/[\X_0]=0=\C_1/[\X_1]
	\]
	as canonical pretriangulated categories. On the other hand, if $S=\Lambda/(\varepsilon)$, the nonsplit exact sequence
	\[
	0\longrightarrow S\longrightarrow \Lambda\longrightarrow S\longrightarrow0
	\]
	shows that $\mathbb E(S,S)=\operatorname{Ext}^1_\Lambda(S,S)\neq0$. Thus $\X_1$ is not rigid. The two quotient categories therefore carry identical quotient-internal information while the ambient cluster tilting behaviour is different.
\end{proof}

\section{Endpoint consequences and comparison}\label{sec:unification}

We conclude by placing the zero and invertible endpoints in the same
extriangulated framework. The zero endpoint is the degree one consequence of
the recognition mechanism. The stable endpoint is the established
mutation/Auslander--Reiten criterion of Zhou--Zhu
\cite[Theorem 4.3]{ZZtri}.
Accordingly, this section compares the two endpoints and summarizes their
relationship. The stable endpoint is included only for comparison and is not
claimed as a new result.

\begin{corollary}\label{cor:zero-endpoint}
Let $(\C,\E,\mathfrak s)$ be an extriangulated category, and let $\X=\add\X$ be a strongly functorially finite rigid subcategory. Then
\[
        \X\text{ is cluster tilting}
        \quad\Longleftrightarrow\quad
        \SigX\cong\zero
        \quad\Longleftrightarrow\quad
        \OmX\cong\zero.
\]
Whenever the hypotheses of an established abelian quotient theorem are satisfied, this endpoint has an abelian ideal quotient. In particular, if $\C$ has enough projective and enough injective objects, then $\C/[\X]$ is abelian by Zhou--Zhu \cite[Corollary 3.5]{ZZct}.
\end{corollary}

\begin{proof}
By Lemma~\ref{lem:adjoint-zero}, we obtain
$\SigX\cong\zero$ if and only if $\OmX\cong\zero$, since
$\SigX$ is left adjoint to $\OmX$. Suppose first that $\X$ is cluster tilting. For $A\in\C$, choose a left $\X$-approximation $\E$-triangle
\[
 A\longrightarrow X_A\longrightarrow A^+\dashrightarrow.
\]
For every $X\in\X$, the standard exact sequence associated with this $\E$-triangle contains
\[
 \C(X_A,X)\longrightarrow\C(A,X)
 \longrightarrow\E(A^+,X)\longrightarrow\E(X_A,X).
\]
The first map is surjective and the last term vanishes by rigidity. Hence $\E(A^+,\X)=0$, so $A^+\in\X$ by cluster tilting. Thus $\SigX(\piX A)=\piX(A^+)=0$ for every $A$, and Lemma~\ref{lem:objectwise-zero} gives $\SigX\cong\zero$.

Conversely, assume $\SigX\cong\zero$ and let $A$ satisfy $\E(\X,A)=0$. In the same approximation $\E$-triangle, $\piX(A^+)=0$. Hence $A^+\in\X$ by Lemma~\ref{lem:zero-object-quotient}. Its extension class belongs to $\E(A^+,A)=0$, so the $\E$-triangle splits and $A$ is a direct summand of $X_A$. Therefore $A\in\X$, and rigidity gives
\[
        \X=\X^{\perp_1}.
\]
Since $\OmX\cong\zero$, the dual argument gives $\X={}^{\perp_1}\X$. Together with strong functorial finiteness, these equalities are precisely the cluster tilting condition. The final abelian assertion is the cited theorem.
\end{proof}

The two endpoints can therefore be stated in terms of the same canonical suspension.

\begin{theorem}\label{thm:unified-endpoints}
Let $(\C,\E,\mathfrak s)$ be an extriangulated category and let $\X=\add\X$ be strongly functorially finite. Equip $\QX$ with its canonical pretriangulated structure.
\begin{enumerate}[{\rm(1)}]
\item \textup{\bf (Zero endpoint)} If $\X$ is rigid, then
\[
        \SigX\cong\zero
        \quad\Longleftrightarrow\quad
        \OmX\cong\zero
        \quad\Longleftrightarrow\quad
        \X\text{ is cluster tilting}.
\]
Under the hypotheses of the standard extriangulated abelian quotient theorem, this implies that $\QX$ is abelian. For example, this is the theorem of Zhou--Zhu \cite[Corollary 3.5]{ZZct} when $\C$ has enough projective and injective objects.

\item \textup{\bf (Stable endpoint)} Assume in addition the Auslander--Reiten duality hypotheses of Zhou--Zhu \cite[Theorem 4.3]{ZZtri}. Then the following conditions are equivalent:
\begin{enumerate}[{\rm(a)}]
\item $(\C,\C)$ is an $\X$-mutation pair.
\item the canonical pretriangulated category $\QX$ is triangulated.
\item $\tau\underline{\X}=\overline{\X}$.
\item the canonical suspension $\SigX$ is an autoequivalence.
\end{enumerate}

\end{enumerate}
Consequently, within extriangulated categories, the cluster tilting and stable regimes are detected by the two extreme behaviours
\[
        \SigX\cong0
        \hspace{2.5mm}\text{and}\hspace{2.5mm}
        \SigX\text{ invertible}
\]
of one canonical endofunctor on one ideal quotient.  The comparison is summarized by
\[
\begin{array}{c@{\qquad\qquad\qquad}c}
\SigX\cong0 & \SigX\text{ invertible}\\[1mm]
\Updownarrow & \Updownarrow {\scriptstyle (\mathrm{AR})}\\[1mm]
\X\text{ cluster tilting} & (\C,\C)\text{ is an }\X\text{-mutation pair}\\[1mm]
\Downarrow\ {\scriptstyle (\mathrm{AQ})} & \Updownarrow\ {\scriptstyle (\mathrm{AR})}\\[1mm]
\QX\text{ abelian} & \QX\text{ triangulated}.
\end{array}
\]
Here \textup{(AQ)} denotes the additional hypotheses of the applicable extriangulated abelian quotient theorem, for instance the enough projectives and enough injectives hypothesis of Zhou--Zhu \cite[Corollary 3.5]{ZZct}. The symbol \textup{(AR)} denotes the Auslander--Reiten duality hypotheses of Zhou--Zhu \cite[Theorem 4.3]{ZZtri}. Thus only the upper equivalence in the left column is unconditional under the assumptions of part \textup{(1)}. The abelian conclusion is a cited consequence under \textup{(AQ)}.
\end{theorem}

\begin{proof}
The equivalences in part \textup{(1)} are Corollary~\ref{cor:zero-endpoint}. The abelian assertion is the cited quotient theorem under \textup{(AQ)}. For part \textup{(2)}, the equivalence of \textup{(a)}--\textup{(c)} is due to Zhou--Zhu \cite[Theorem 4.3]{ZZtri}. Condition \textup{(b)} implies \textup{(d)} because the suspension of a triangulated category is an autoequivalence, while \textup{(d)} implies \textup{(b)} by Lemma~\ref{lem:stable-right} applied to the canonical right triangulated quotient.
\end{proof}

\begin{remark}\label{rem:scope-endpoints}
\begin{enumerate}[{\rm(1)}]
	\item In degree one, the zero endpoint and stable endpoint use only the ordinary extriangulated extension bifunctor and therefore require neither smallness nor $R$-linearity.
	
	\item The left column of Theorem~\ref{thm:unified-endpoints} is specific to the canonical ideal quotient and its ambient extriangulated structure. It is not a formal statement that a pretriangulated category with zero suspension is abelian. Indeed, in the formalism of Beligiannis--Reiten \cite[Chapter II, Section 1]{BR}, an additive category with kernels and cokernels carries a pretriangulated structure with $\Sigma=\Omega=0$, without being abelian in general. The content of the zero endpoint is therefore the recognition
	\[
	\SigX\cong0
	\quad\Longrightarrow\quad
	\X\text{ cluster tilting}
	\]
	under rigidity and approximation hypotheses, followed, under \textup{(AQ)}, by the established abelian quotient theorem.  By contrast, invertibility of the suspension of a right triangulated category is already sufficient to recover a triangulated structure by Lemma~\ref{lem:stable-right}.
\end{enumerate}
\end{remark}

Both branches of Theorem~\ref{thm:unified-endpoints} are formulated in the extriangulated ambient category. The triangulated case is the following specialization.

\begin{corollary}\label{cor:triangulated-specialization}
Let $\T$ be a triangulated category with Auslander--Reiten translation $\tau$, and let $\X=\add\X$ be a functorially finite rigid subcategory. Denote by $\Sigma_{\X}$ the canonical suspension of the pretriangulated ideal quotient $\T/[\X]$.
\begin{enumerate}[{\rm(1)}]
\item
\[
 \X\text{ is cluster tilting}
 \quad\Longleftrightarrow\quad
 \Sigma_{\X}\cong0.
\]
In this case $\T/[\X]$ is abelian by Koenig--Zhu \cite[Theorem 3.3]{KZ}.
See also Zhou--Zhu \cite[Corollary 3.6]{ZZct}.
\item
\[
 \tau\X=\X
 \quad\Longleftrightarrow\quad
 \T/[\X]\text{ is triangulated}
 \quad\Longleftrightarrow\quad
 \Sigma_{\X}\text{ is an autoequivalence}.
\]
The first equivalence is J\o rgensen's theorem \cite[Theorem 3.3]{Jor}, in the formulation of Zhou--Zhu \cite[Corollary 4.4]{ZZtri}. The equivalence with invertibility of $\Sigma_{\X}$ follows from the canonical right triangulated structure and Lemma~\ref{lem:stable-right}.
\end{enumerate}
Hence the familiar triangulated diagram
\[
\begin{array}{c@{\qquad\qquad\qquad}c}
\Sigma_{\X}\cong0 & \Sigma_{\X}\text{ invertible}\\[1mm]
\Updownarrow & \Updownarrow\\[1mm]
\X\text{ cluster tilting} & \tau\X=\X\\[1mm]
\Downarrow & \Updownarrow\\[1mm]
\T/[\X]\text{ abelian} & \T/[\X]\text{ triangulated}
\end{array}
\]
is a special case of the extriangulated theorem above.
\end{corollary}

For $n>1$ the left endpoint admits the higher form established in Theorem~\ref{thm:recognition}:
\[
        \SigX^n\cong0
        \quad\Longleftrightarrow\quad
        \X=
        \bigcap_{i=1}^{n}{}^{\perp_i}\X
        =
        \bigcap_{i=1}^{n}\X^{\perp_i}.
\]
Thus the relevant structural axis is \emph{truncation versus stability}. Finite nilpotency detects two-sided maximal higher orthogonality, whereas invertibility detects the mutation/AR stable regime. For a general subcategory $\X$, the canonical suspension need be neither nilpotent nor invertible.

\begin{remark}
If an endofunctor $F$ of an additive category satisfies $F^p\cong\operatorname{Id}$ for some $p\geq1$, then $F^{p-1}$ is an inverse equivalence. If, in addition, $F^m\cong\zero$ for some $m\geq1$, composing with an inverse of $F^m$ gives $\operatorname{Id}\cong\zero$, so the category is zero. Applied to $F=\SigX$, this shows that a nonzero canonical quotient cannot be both truncated and periodic.
\end{remark}

\medskip
\noindent\textbf{Acknowledgements:} Panyue Zhou is supported by the National Natural Science Foundation of China (Grant No.~12371034).
\vspace{2mm}

\noindent{\bf Statement on AI:} ChatGPT was used only for language editing and improving the clarity of the manuscript. All mathematical results, arguments, and proofs were developed and verified by the authors.
\vspace{2mm}

\noindent\textbf{Data Availability:} Data sharing not applicable to this article as no datasets were generated or analysed during
the current study.
\vspace{2mm}

\noindent\textbf{Conflict of Interests:} The authors declare that they have no conflicts of interest to this work.

\end{document}